\documentclass{proc-l}
\usepackage{amssymb,amsmath,amsthm}
\usepackage[english]{babel}
\usepackage{t1enc}
\usepackage[latin2]{inputenc}
\usepackage{epsfig}

\newtheorem{thm}{Theorem}[section]

\newtheorem{lemma}[thm]{Lemma}

\newtheorem{claim}[thm]{Claim}

\def\mathbf{\boldsymbol}

\usepackage{indentfirst}
\begin{document}

\title{A simple proof of the Zeilberger--Bressoud $q$-Dyson theorem}
\author[Gy. K\'arolyi]{Gyula K\'arolyi}

\address{School of Mathematics and Physics,
The University of Queensland, Brisbane, QLD 4072, Australia}
\email{karolyi@cs.elte.hu}
\thanks{This research was supported by the Australian Research Council,
by ERC Advanced Research Grant No. 267165, and by
Hungarian National Scientific
Research Funds (OTKA) Grants 67676 and 81310.}

\author[Z. L. Nagy]{Zolt\'an L\'or\'ant Nagy}
\address{Alfr\'ed R\'enyi Institute of Mathematics, 
Re\'altanoda utca 13--15, Budapest, 1053 Hungary}
\email{nagyzoltanlorant@gmail.com}






\begin{abstract}
\noindent
As an application of the Combinatorial Nullstellensatz, we
give a short polynomial proof of the $q$-analogue of Dyson's conjecture
formulated by Andrews and first proved by Zeilberger and Bressoud.
\end{abstract}

\maketitle

\section{Introduction}

Let $x_1,\ldots,x_n$ denote independent variables, each associated with
a nonnegative integer $a_i$. Motivated by a problem in statistical physics
Dyson \cite{Dyson} in 1962 formulated the hypothesis that the constant term of 
the Laurent polynomial
$$ \prod_{1\leq i \neq j \leq n}\left(1-\frac{x_i}{x_j}\right)^{a_i}$$ 
is equal to the multinomial 
coefficient $(a_1+a_2+\dots +a_n)!/(a_1!a_2! \dots a_n!)$.
Independently Gunson [unpublished] 
and Wilson \cite{Wilson} confirmed the statement in the same year, 
then Good gave an elegant proof \cite{Good} using Lagrange interpolation.
\bigskip

Let $q$ denote yet another independent variable. In 1975
Andrews \cite{Andrews} suggested the following $q$-analogue 
of Dyson's conjecture: The constant term 
of the Laurent polynomial
$$f_q(\mathbf{x}):=
f_q(x_1, x_2, \dots, x_n)=\prod_{1\leq i < j \leq n} 
\left(\frac{x_i}{x_j}\right)_{a_i}\left(\frac{qx_j}{x_i}\right)_{a_j}
\in \mathbb{Q}(q)[\mathbf{x},\mathbf{x}^{-1}]$$  
must be  
$$\frac{\left(q\right)_{a_1+a_2+\dots +a_n}}
{\left(q\right)_{a_1}\left(q\right)_{a_2}\dots\left(q\right)_{a_n}},$$
where $\big(t\big)_{k}= (1-t)(1-tq)\dots(1-tq^{k-1})$ with 
$\big(t\big)_{0}$ defined to be $1$.  
Specializing at $q=1$, Andrews' conjecture gives back that of Dyson.

Despite several attempts \cite{Kadell,Stanley,Stanley2} the problem 
remained unsolved until 1985, when Zeilberger and Bressoud \cite{Zeilberger2} 
found a 
combinatorial proof. Shorter proofs for the equal parameter case 
$a_1=a_2=\ldots=a_n$ are due to Habsieger \cite{Habsieger}, Kadell
\cite{Kadell2} and Stembridge \cite{Stembridge}; they cover
the special case $A_{n-1}$ of a problem of Macdonald \cite{Macdonald}
concerning root systems, which was solved in full generality by Cherednik
\cite{Cherednik}.
A shorter proof of the Zeilberger--Bressoud theorem,
manipulating formal Laurent series,
was given by Gessel and Xin \cite{Gessel}.
\bigskip

Following up a recent idea of Karasev and Petrov 
we present a very short combinatorial proof using polynomial techniques. 
We find that their proof of the Dyson conjecture in \cite{Karasev}
naturally extends for Andrews' $q$-Dyson conjecture.
We note that built on the same basic principles but with more
sophisticated details it is possible to prove a whole family of constant
term identities for Laurent polynomials, including 
the Bressoud--Goulden theorems \cite{Bressoud},
conjectures of Kadell \cite{Kadell3,Kadell4},
the $q$-Morris constant
term identity \cite{Habsieger,Kadell2,Morris,Zeilberger} and its far
reaching generalizations conjectured by Forrester \cite{Baker,Forrester};
see \cite{Karolyi,Karolyi2,Karolyi3}. 
We decided to publish this proof separately because of its sheer simplicity.

\section{The proof}

Note that if $a_i=0$, then we may omit all factors that include the
variable $x_i$ without affecting the constant term of $f_q$.
Accordingly, we may assume that each $a_i$ is a positive integer.
Consider the homogeneous polynomial 
$$F(x_1, x_2, \dots, x_n)= 
\prod_{1\leq i < j \leq n}
{\left( 
\prod_{t=0}^{a_i-1}{(x_j-x_iq^t)}\cdot
\prod_{t=1}^{a_j}{(x_i-x_jq^t)}
\right)}\in \mathbb{Q}(q)[\mathbf{x}].$$ 
Clearly, the constant term of $f_q(\mathbf{x})$ is equal to the coefficient 
of $\prod_i{x_i^{\sigma-a_i}}$ in $F(\mathbf{x})$, where $\sigma=\sum_ia_i$.
To express this coefficient we apply the following effective version of the
Combinatorial Nullstellensatz \cite{Alon} observed independently by
Laso\'n \cite{Lason} and by Karasev and Petrov \cite{Karasev}.
A sketch of the proof is included for the sake of completeness. 

\begin{lemma}\label{interpol} 
Let $\mathbb{F}$ be an arbitrary field and 
$F\in \mathbb{F}[x_1, x_2, \dots, x_n]$ a polynomial of degree
$\deg(F)\leq d_1+d_2+\dots+d_n$. 
For arbitrary subsets $A_1, A_2, \dots, A_n$ of $\mathbb{F}$ 
with $|A_i|=d_i+1$, the coefficient of $\prod x_i^{d_i}$ in $F$ is
$$ \sum_{c_1\in A_1} \sum_{c_2\in A_2} \dots \sum_{c_n\in A_n} 
\frac{F(c_1, c_2, \dots, c_n)}{\phi_1'(c_1)\phi_2'(c_2)\dots \phi_n'(c_n)},$$
where $\phi_i(z)= \prod_{a\in A_i}(z-a)$.
\end{lemma}

\begin{proof}
Construct a sequence of polynomials 
$F_0:=F, F_1,\ldots,F_n\in \mathbb{F}[\mathbf{x}]$ recursively
as follows. For $i=1,\ldots,n$, let $F_i=F_i(\mathbf{x})$ denote the
remainder obtained after dividing $F_{i-1}(\mathbf{x})$ by $\phi_i(x_i)$ over the
ring $\mathbb{F}[x_1,\ldots,x_{i-1},x_{i+1},\dots,x_n]$. This process does not 
affect the coefficient of $\prod x_i^{d_i}$. The polynomial $F_n$ satisfies
$F_n(\mathbf{c})=F(\mathbf{c})$ for all $\mathbf{c}\in A_1\times \dots \times
A_n$ and its degree in $x_i$ is at most $d_i$ for every $i$. The unique 
polynomial with that property is expressed in the form
$$F_n(\mathbf{x})=\sum_{\mathbf{c}\in A_1\times\dots\times A_n}F(\mathbf{c})
\prod_{i=1}^n\prod_{\substack{\gamma\in A_i\\\gamma\ne c_i}}\frac{x_i-\gamma}{c_i-\gamma}$$
by the Lagrange interpolation formula, hence the result.
\end{proof}

The idea is to apply this lemma taking $\mathbb{F}=\mathbb{Q}(q)$ with a 
suitable choice of the sets $A_i$ such that $F(\mathbf{c})=0$ for all but
one element $\mathbf{c}\in A_1\times\dots\times A_n$. 
Put $A_i=\{1, q, \dots, q^{\sigma-a_i}\}$, then $|A_i|=\sigma-a_i+1$; 
and introduce
$\sigma_i=\sum_{j=1}^{i-1}{a_j}$. Thus, $\sigma_1=0$ and $\sigma_{n+1}=\sigma$.

\begin{claim}
For $\mathbf{c}\in A_1\times\dots\times A_n$ we have $F(\mathbf{c})=0$,
unless $c_i=q^{\sigma_i}$ for all $i$.
\end{claim}

\begin{proof}
Suppose that $F(\mathbf{c})\ne 0$ for the numbers $c_i=q^{\alpha_i}\in A_i$.
Here $\alpha_i$ is an integer satisfying $0\le \alpha_i\le \sigma-a_i$.
Then for each pair $j>i$, either $\alpha_j-\alpha_i\ge a_i$, or 
$\alpha_i-\alpha_j\ge a_j+1$. In other words, $\alpha_j-\alpha_i\ge a_i$
holds for every pair $j\ne i$, with strict inequality if $j<i$.
In particular, all of the $\alpha_i$ are distinct.
Consider the unique permutation $\pi$ satisfying 
$\alpha_{\pi(1)}< \alpha_{\pi(2)}<\dots< \alpha_{\pi(n)}$.
Adding up the inequalities $\alpha_{\pi(i+1)}-\alpha_{\pi(i)}\ge a_{\pi(i)}$
for $i=1,2\ldots,n-1$ we obtain 
$$\alpha_{\pi(n)}-\alpha_{\pi(1)}\ge 
\sum_{i=1}^{n-1}a_{\pi(i)}=\sigma-a_{\pi(n)}.$$
Given that $\alpha_{\pi(1)}\ge 0$ and 
$\alpha_{\pi(n)}\le \sigma-a_{\pi(n)}$, strict
inequality is excluded in all of these inequalities. It follows that
$\pi$ must be the identity permutation and
$\alpha_i=\alpha_{\pi(i)}=\sum_{j=1}^{i-1}a_{\pi(j)}=\sigma_i$
must hold for every $i=1,2,\dots,n$.
This proves the claim.
\end{proof}

This way finding the constant term of $f_q$ is reduced to the evaluation of
$$\frac{F(q^{\sigma_1}, q^{\sigma_2}, \dots, q^{\sigma_n})}
{\phi_1'(q^{\sigma_1})\phi_2'(q^{\sigma_2})\dots \phi_n'(q^{\sigma_n})},$$
where $\phi_i(z)=(z-1)(z-q)\dots(z-q^{\sigma-a_i})$. Here

\begin{align*}
\phi_i'(q^{\sigma_i})&=
\prod_{t=0}^{\sigma_i-1}{(q^{\sigma_i}-q^t)}\cdot
\prod_{t=\sigma_i+1}^{\sigma-a_i}{(q^{\sigma_i}-q^t)}\\ 
&=\prod_{t=0}^{\sigma_i-1}{q^t(q^{\sigma_i-t}-1)}\cdot
\prod_{t=1}^{\sigma-\sigma_{i+1}}{q^{\sigma_i}(1-q^t)}\\
&=(-1)^{\sigma_i}q^{\tau_i}
\left(q\right)_{\sigma_i}\left(q\right)_{\sigma-\sigma_{i+1}}
\end{align*}

\noindent
with $\tau_i=\binom{\sigma_i}{2}+\sigma_i(\sigma-\sigma_{i+1})$, whereas

\begin{align*}
{F(q^{\sigma_1}, q^{\sigma_2}, \dots, q^{\sigma_n})}&=
\prod_{1\leq i < j \leq n}\left(
\prod_{t=0}^{a_i-1}{q^{\sigma_i+t}(q^{\sigma_j-\sigma_i-t}-1)} \cdot
\prod_{t=1}^{a_j}{q^{\sigma_i}(1-q^{\sigma_j-\sigma_i+t})} 
\right)\\
&=(-1)^uq^v\prod_{1\leq i < j \leq n}\left(
\frac{\left(q\right)_{\sigma_j-\sigma_i}}
{\left(q\right)_{\sigma_j-\sigma_{i+1}}}
\cdot
\frac{\left(q\right)_{\sigma_{j+1}-\sigma_i}}
{\left(q\right)_{\sigma_j-\sigma_{i}}}\right)\\
&=(-1)^uq^v
\prod_{i=1}^n \frac{\left(q\right)_{\sigma_i}\left(q\right)_{\sigma-\sigma_i}}
{\left(q\right)_{\sigma_{i+1}-\sigma_i}}
\end{align*}

\noindent 
with $u=\sum_i(n-i)a_i$ and $v=\sum_i\left((n-i)a_i\sigma_i+
(n-i)\binom{a_i}{2}+\sigma_i(\sigma-\sigma_{i+1})\right)$.
\bigskip

In view of the simple identity $\sum_i(n-i)a_i=\sum_i\sigma_i$, we have
$u=\sum_i\sigma_i$, thus the powers of $-1$ cancel out. The same happens
with the powers of $q$ due to the following observation, which implies
$v=\sum_i\tau_i$.

\begin{claim}
$\sum_i{(n-i)\left(a_i\sigma_i+\binom{a_i}{2}\right)}= 
\sum_i\binom{\sigma_i}{2}.$
\end{claim}

\begin{proof}
We proceed by a routine induction on $n$. 
When $n=0$, both expressions are 0, and
one readily checks the relation
$$\sum_{i=1}^n\left(a_i\sigma_i+\binom{a_i}{2}\right)
=\binom{\sigma_{n+1}}{2},$$
which completes the induction.
\end{proof}

Putting everything together we obtain that the constant term of $f_q$ is indeed

\begin{align*}
\frac{F(q^{\sigma_1}, q^{\sigma_2}, \dots, q^{\sigma_n})}
{\phi_1'(q^{\sigma_1})\phi_2'(q^{\sigma_2})\dots \phi_n'(q^{\sigma_n})}&=
\prod_{i=1}^n
\frac{\left(q\right)_{\sigma_i}\left(q\right)_{\sigma-\sigma_i}}
{\left(q\right)_{\sigma_i}\left(q\right)_{\sigma-\sigma_{i+1}}
\left(q\right)_{\sigma_{i+1}-\sigma_i}}\\
&=\frac{\left(q\right)_{\sigma}}
{\displaystyle{\prod_{i=1}^n\left(q\right)_{\sigma_{i+1}-\sigma_i}}}\\
&=\frac{\left(q\right)_{a_1+a_2+\dots +a_n}}
{\left(q\right)_{a_1}\left(q\right)_{a_2}\dots\left(q\right)_{a_n}}.
\end{align*}

\bibliographystyle{amsplain}

\end{document}